\documentclass[11pt]{article}

\usepackage{amssymb, verbatim, amsmath}
\usepackage[all]{xy}
\usepackage{amsthm}
\usepackage[utf8]{inputenc}
\parskip \baselineskip

\newcommand{\bc}{\begin{center}}
\newcommand{\ec}{\end{center}}
\newcommand{\be}{\begin{enumerate}}
\newcommand{\ee}{\end{enumerate}}
\newcommand{\beq}{\begin{equation}}
\newcommand{\eeq}{\end{equation}}
\newcommand{\bi}{\begin{itemize}}
\newcommand{\ei}{\end{itemize}}
\newcommand{\bd}{\begin{description}}
\newcommand{\ed}{\end{description}}
\newcommand{\ba}{\begin{array}}
\newcommand{\bea}{\begin{eqnarray*}}
\newcommand{\eea}{\end{eqnarray*}}
\newcommand{\ea}{\end{array}}
\newcommand{\bt}{\begin{tabular}}
\newcommand{\et}{\end{tabular}}

\newcommand{\bmi}{\begin{minipage}}
\newcommand{\emi}{\end{minipage}}

\newtheorem{thm}{Theorem}[section]
\newtheorem{defn}[thm]{Definition}
\newtheorem{lem}[thm]{Lemma}
\newtheorem{pro}[thm]{Proposition}

\newtheorem{exa}[thm]{Example}

\newtheorem{rem}[thm]{Remark}

\begin{document}

\bc {\bf\large The distributive elements of a near-field}\\[3mm]
{\sc Julien Bahimuzi }

\it\small
African Institute for Mathematical sciences\\
AIMS RWANDA\\
\rm e-mail: julien.bahimuzi@aims.ac.rw
\ec
 
\normalsize

\quotation{\small {\bf Abstract:} 
\small
{In this thesis, we investigated some properties of (left)near fields  and  derived  some results. We are focusing on $D(\alpha, \beta)$ which is the generalized set of distributive elements of a nearfield. In particular, we investigated some conditions on $\alpha,\beta, \alpha+\beta$ for $D(\alpha, \beta)$ to be a subfield of $\mathbb{F}_{q^{n}}$. In nearfield theory the two distributive laws  can not hold at the same time. So in term of left nearfield and near-ring,  the right distributivity does not hold and to solve the problem, we defined a set of all distributive elements called $D(R)$.} \\
\normalsize

%\textup{2023} \textit{MSC}: \textup{16Y30;12K05}

\section{Introduction}
The study of distributive elements in nearfields has a rich history, dating back to the early 20th century. The idea of a near field is introduced in $1905$ where the American mathematician L. Dickson examined it for the first time and presented a first example of a nearfield (\cite{hussein2022some}). The concept of distributivity was introduced by Wedderburn in his classic paper on quasifields (\cite{Wedderburn1926OnQ}). He defined a quasifield as an algebraic system in which the multiplication operation distributes over addition. Later, Bruck (\cite{Bruck1946Contributions}) extended this concept to nearfields, which are more general than quasifields.

Several researchers have studied distributive elements in nearfields and their applications. In particular, Dickson nearfields have received considerable attention due to their interesting algebraic properties and their connection with coding theory and cryptography. Dickson nearfields are a family of nearfields that are constructed using finite fields and a quadratic form (\cite{Dickson1958Linear}). The distributive elements of a Dickson nearfield are closely related to the quadratic residues and non-residues of the finite field (\cite{Gow2017Distributive}).

In recent years, there has been renewed interest in the study of distributive elements in nearfields, due to their applications in combinatorial designs and cryptography. Several families of nearfields have been constructed using distributive elements, such as twisted nearfields and generalized nearfields [(\cite{Clay1974Twisted}), (\cite{Kinyon2013Generalized})]. These families have been used to construct efficient error-correcting codes and cryptographic primitives.

Despite the extensive research on distributive elements in nearfields, several open problems remain. For example, it is still an open question whether every finite nearfield has a non-trivial distributive element (\cite{Minev1998Distributive}). Also, the structure and properties of the generalized set of distributive elements are not well-understood, and further investigation is needed.

Recently the notions of near vector space have been defined in (\cite{djagba2019contributions}) and in (\cite{djagba2020subspace}) the characterized  subspace structure of Beidleman near-vector spaces is investigated.

spaces and classify their R-subgroups. 
The main contribution of this thesis is to provide a comprehensive study of the generalized set of distributive elements in nearfields. We investigate the structure and properties of this set, and provide some  results and insights. Our study sheds light on the behavior of distributive elements in nearfields, and provides a basis for further research in this area.

\subsection{Motivation}
Nearfields are important algebraic structures that have been extensively studied due to their applications in coding theory, cryptography, and combinatorial designs. They generalize both fields and quasifields, and their properties and structures are of great interest to mathematicians and engineers alike. One important property of nearfields is the distributivity of their multiplication operation over addition, which has led to the study of distributive elements in nearfields.
In this thesis we are going to study the set $D(\alpha, \beta)$ and investigate some condition on $\alpha$ and $\beta$ so that $D(\alpha,\beta)$ can be a subflied of $\mathbb{F}_{q^{n}}$ where $(q, n)$ is a Dickson pair.
\subsection{Research problem}

The distributive elements of a nearfield have been studied by many researchers, but there are still several open problems and unanswered questions regarding their properties and behavior. In particular, the structure and properties of the generalized set of distributive elements, which is a subset of the nearfield that contains all distributive elements, have not been fully understood.
\subsection{Research objectives}

The distributive elements of a nearfield have been studied by many researchers, but there are still several open problems and unanswered questions regarding their properties and behavior. In particular, the structure and properties of the generalized set of distributive elements, which is a subset of the nearfield that contains all distributive elements, have not been fully understood.
\subsection{Reseach objectives}
The main objective of this thesis is to investigate the properties and behavior of distributive elements in nearfields, with a focus on the generalized set of distributive elements. Specifically, we aim to:
\begin{itemize}
	\item[$\bullet$] Define and characterize the generalized set of distributive elements.
	\item[$\bullet$]Study the structure and properties of this set.
	
	\item[$\bullet$]Provide some insights and results related to distributive elements in nearfields.
\end{itemize}

\subsection{Plan of the thesis}
The thesis is structured as follows:

\begin{itemize}
	\item[$\bullet$] \textbf{Chapter 2} provides the necessary background and definitions related to nearfields and distributive elements.
	\item[$\bullet$]	\textbf{Chapter 3} constructs a finite Dickson nearfield, which serves as an important example for our study.
	\item[$\bullet$]\textbf{Chapter 4} defines the generalized set of distributive elements and studies its properties.
	\item[$\bullet$]	\textbf{Chapter 5} Gives some  details on  applications in codding theory.
	Finally, 
	\item[$\bullet$] \textbf{Chapter 6} concludes our study and suggests some directions for future research.
\end{itemize}
Nearfields and near-rings are related to many other structures and needed for
several representation theorems. Therefore it is important to gain knowledge about
the structure of near-rings and nearfields and to find construction methods. The first
examples of proper nearfields were constructed by L.E. Dickson 1905, they were
finite (\cite{karzel1980some}).
\newpage

\section{Definitions and preliminary results}
In this chapter, we are going to  give some important definitions on 
nearfields and present some important  results. We begin by defining some elementary structures. These are standard definitions. They can be found in most elementary algebra books. A nearfield is considered by Dickson as a field with only  one distributive law. Therefore, we start by introducing fields and nearfiels.
\subsection{Fields  and nearfields}
\begin{defn}
	A \emph{field} $ \mathbb{F}$ is defined by giving a set $ \mathbb{F} $ with two  binary operations $ "+" $ and $ "\cdot" $ on $ \mathbb{F}$,	i.e two maps \begin{align*}
	& +:\mathbb{F}\times \mathbb{F}\longrightarrow \mathbb{F},\\
	& \cdot :\mathbb{F}\times \mathbb{F}\longrightarrow \mathbb{F},
	\end{align*}
	subject to the the axioms (\cite{murphy2006course}):
	\begin{itemize}
		\item[$ (A_{1}) $]  Addition is associative,
		
		for all $ a, b, c \in \mathbb{F} $, $ (a+b)+c=a+(b+c).$
		\item[($ A_2 $)] Addition is commutative,
		
		for all   $ a, b \in \mathbb{F} $, $a+b=b+a.$
		\item[$ (A_{3}) $]  Addition  has an identity element,
		
		for all $ a\in \mathbb{F} $, $ a+0=0+a=a.$
		\item[$ (A_{4}) $] Each element has its inverse with respect to addition,
		
		for all $ a \in \mathbb{F} $, $\exists b\in \mathbb{F}$ such that $a+b=0.$
		\item[$ (A_{5}) $]  Multiplication  is associative,
		
		for all $ a, b, c \in \mathbb{F} $, $ (ab)c=a(bc).$
		\item[$ (A_{6}) $]  Multiplication is commutative,
		
		for all $ a, b \in \mathbb{F} $, $ ab=b.a$
		\item[$ (A_{7}) $]  Multiplication has an identity element,
		
		for all $ a\in \mathbb{F} $, $ a.1=a.$
		\item[$ (A_{8}) $] Each non-zero element has an inverse with respect to multiplication,
		
		for all $ a \in \mathbb{F}^{*} $,  $\exists b  \in \mathbb{F}$ such that $ ab=1.$
		\item[$ (A_{9}) $] Multiplication is distributive over addition,
		
		for all $ a, b, c \in \mathbb{F} $, $ a(b+c)=ab+ac.$
	\end{itemize}
\end{defn}
\newpage

\begin{pro}\cite{murphy2006course}

	Suppose $\mathbb{F}$ is a field. Then 
	\begin{enumerate}
		\item For each $ a, b\in \mathbb{F}$, the equation $ a+x=b $ has a unique solution.
		
		\item For each $ a, b\in \mathbb{F}$, the equation $ ay=b $ has a unique solution.
	\end{enumerate}
\end{pro}
%\begin{proof}
%\begin{enumerate}
%\item Using $ (A_{4}) $, for all $a\in \mathbb{F}$ there exists $  c  \in \mathbb{F}$ such that $ a+c=0.$
%
%Then
%\begin{align*}
%a+(b+c)&=(a+c)+ b\quad (A_{1}) \\
%&=0+b\\
%&= b \quad (A_{3} ). 
%\end{align*}
%Thus $ x=c+b $ is a solution of the equation  $ a+x=b $. It is moreover the only solution. 
%
%If $x, y\in \mathbb{F}$ are two solutions, then 
%\begin{align*}
%a+x&=b=a+y \\
%\Rightarrow c+(a+x)&= c+(a+y)\\
%\Rightarrow (a+c)+x&= (c+a)+y\quad (A_{1})\\
%\Rightarrow (c+a)+x&= (a+c)+y\quad (A_{2})\\
%\Rightarrow 0+x&=0+y\\
%\Rightarrow x&=y.
%\end{align*}
%\item  Using $  (A_{9})$, for all $a\mathbb{F}^{*}$, there exists $ c\in \mathbb{F}$ such that  $ac=1.$
%Then  
%\begin{align*}
%(ab)c&=(ac)b \quad (A_{5})\\
%&=1.b\\
%&=b\qquad (A_{7}).
%\end{align*}
%Thus $y=cb$ is a  solution of the equation  $  ay=b$. It is moreover the only solution.
%If $x, y\mathbb{F}$ are two solutions, then 
%\begin{align*}
%ay&=b=ax \\
%\Rightarrow acy&=cax\\
%\Rightarrow (ac)y&=(ac)x\\
%\Rightarrow 1.y&=1.x\\
%\Rightarrow y&=x\\
%\Rightarrow x&=y.
%\end{align*} Therefore, both the equations above have a unique solution.
%\end{enumerate} 
%\end{proof}
In the next sections, we will need to define a finite field. Therefore we consider the following definitions.
\begin{defn}
	$\mathbb{F}$ is a \emph{finite field} or \emph{Galois field}  if it is a field that contains a finite number of elements. As with any field, a finite field is a set on which the operations of multiplication and addition are defined and satisfy certain basic rules as $(A_{1})$  to $(A_{9})$ (\cite{mullen2013handbook}).
	The most common examples of finite fields are given by $\mathbb{Z}/p\mathbb{Z}$ when p is a prime number. We have the following
	\begin{thm}
		Suppose $ \mathbb{F} $ is a finite field of characteristic $p$. Then $ \mathbb{F} $ contains $p^{n}$ elements for some n: $\mid \mathbb{F}\mid= p^{n}$ (\cite{murphy2006course}).
	\end{thm}
	\begin{proof}
		Let us suppose that $\mathbb{F}$ has dimension $n$ over $p$. It means that $\mathbb{F}$ is considered as a vector space. Then we can find a basis 
		\begin{equation}\label{basis}
		\left\lbrace e_{1}, e_{2}, \cdots, e_{n} \right\rbrace 
		\end{equation} for $\mathbb{F}$ $p$.
		Every element $x\in \mathbb{F}$ can be expressed as a linear combination  of the basis (\ref{basis}).

		\begin{equation*}
		x=\lambda_{1}e_{1}+\lambda_{2}e_{2}+\dots+\lambda_{n}e_{n}
		\end{equation*}
		There are $p$ choices for each $\lambda_{i}$, so the total number of elements in $\mathbb{F}$ is 
		$ \overbrace{p\cdot p\cdot p\dots p}^{\text{n times}}= p^{n}$ .
		
	\end{proof}
\end{defn}
The order of a finite field is its number of elements, which is either a prime number or a prime power. For every prime number $q$ and every positive integer n there are fields of order 
$q^n$, all of which are isomorphic. The finite field of order $q^n$ is denoted by $\mathbb{F}_{q^{n}}$ (\cite{mullen2013handbook}).

\begin{defn}
	Let us consider the set $ R$ with two binary operations $ + $ and  $\cdot$     denoted  by  $(R, + , \cdot).$ 
	If 
	\begin{itemize}
		\item[(i)] $ (R, + ) $ is a group
		\item[(ii)] $ (R, \cdot ) $ is a semigroup
		\item[(iii)] $ a(b+c)=ab+ac, \qquad \forall a, b, c \in R,$ 
	\end{itemize}  
	then $ (R, + , \cdot) $ is a near-ring (left).
	
	Moreover, if $(R^{*}, \cdot)  $ is a group, then  $(R, + , \cdot) $  is a \emph{nearfield (left)} where $  R^{*}=R-\left\lbrace 0\right\rbrace $ (\cite{djagba2020generalized}). 
\end{defn} 
We abbreviate $ (R,+, \cdot) $ by $ R $ when the operations are clearly understood
and omit the symbol ''$ \cdot$'' for multiplication if no confusion is possible.
We have the following example.
\begin{exa}\label{lemera} 
	Let $ (G,+) $ be a group. Then $ (M(G),+, \circ) $ is a nearring under pointwise
	addition and composition
	\begin{equation*}
	(x)(f+g)=(x)f+(x)g,
	\end{equation*}
	and \begin{equation*}
	(x)(f\circ g)=((x)f)g,
	\end{equation*}
	where
	\begin{equation}
	M(G): =\left\lbrace f: G\longrightarrow  G\right\rbrace 
	\end{equation}
	is the set of all mappings on $G$. 
	And it is easy to show that $ (M(G),+, \circ) $ is a left near ring (the left distributivity holds)(\cite{howell2007contributions}).
	
	\textbf{Claim}: \emph{We want to prove that $M(G)$ is a left near ring}
	\begin{proof}
		Let us consider   three maps $f, g, h$ which belong to $M(G)$.
		\begin{itemize}
			\item[$\bullet$] Let us  define a mapping $(x)\zeta: = 0$ for all $x\in G$, then $\zeta$ is an element of $M(G)$ $\Rightarrow M(G)\neq \emptyset$.
			\item[$\bullet$] Let us show that $M(G)$ is a group.
			\begin{align*}
			(x)((f+g)+h)&=(x)(f+g)+(x)h\quad (\text{by definition})\\
			&=(x)f+(x)g+(x)h\quad (\text{by definition})\\
			&=(x)f+(x)(g+h)\\
			&=(x)(f+(g+g)), \quad \text{for all $ x\in G$}.
			\end{align*}
			We can see that $(M(G), +)$ is a semigroup.
			
			Fr all $f\in M(G)$, there exists $-f\in M(G)$ such that $(x)f+(x)(-f)=0$,  $(x)(-f)= -(x)f$ for all $x\in G$.
			We define the mapping $-f: G\longrightarrow G$ by $(x)(-f)=-(x)f$ for all $x\in G$.
			\begin{align*}
			(x)(\zeta+f)&=(x)\zeta+(x)f\\
			&=0+(x)f\\
			&=(x)f \quad \text{and}\\
			(x)(f+\zeta)&=(x)f+(x)\zeta\\
			&=(x)f+0\\
			&=(x)f.
			\end{align*}
			It follows that
			\begin{align*}
			(x)(f+(-f))&=(x)f-(x)f\\
			&=(x)\zeta\\
			&=0 \quad \text{and}\\
			(x)((-f)+f)&=(x)(-f)+(x)f\\
			&=-(x)f+(x)f\\
			&=(x)\zeta\\
			&=0
			\end{align*}
			Hence for any $f\in M(G)$, $\zeta+f=f+\zeta= f$ and $f+(-f)= (-f)+f= \zeta$.
			Therefore $(M(G), +)$ is a group. 
			\item[$\bullet$] $(M(G), \circ)$ is a semigroup. Then,
			\begin{align*}
			(x)(f\circ(g\circ h))&=((x)f)(g\circ h)\\
			&=((x)f)\left[(g)h \right]\\
			&=\left[((x)f)g \right]h\\
			&=\left[ (x)(f\circ g)\right]h\\
			&=(x)\left[(f\circ g)\circ h \right].   
			\end{align*}
			Hence $(M(G), \circ)$ is a semigroup.
			
			\item[$\bullet$] The composition distributes over pointwise  addition  in one direction in $M(G)$.
			
			For all $f, g, h\in M(G)$
			\begin{align*}
			(x)(f\circ (g+h))&=(x)(f\circ g)+(x)(f\circ h)\\
			\text{In fact,} (x)\left[f\circ (g+h) \right]\\
			&=((x)f)(g+h)\\
			&=((x)f)g+((x)f)h\\
			&=(x)(f\circ g)+ (x)(f\circ h) , \quad \text{for all $x\in G$}.
			\end{align*} Hence the left  distributive low holds. We conclude that $M(G)$ with addition and the function composition $"\circ"$  a near ring.
			In fact the right distributive low failed. We can see that in considering $a, b, c$ all different from zero and for all $x\in G$ we define the maps $h_{a}, h_{b}, h_{c}$ as follow.
			\begin{align*}
			h_{a}:& G\longrightarrow G\\
			&x\longrightarrow (x)h_{a}=a,\\
			h_{b}:& G\longrightarrow G\\
			&x\longrightarrow (x)h_{b}=b,\\
			h_{c}:& G\longrightarrow G\\
			&x\longrightarrow (x)h_{c}=c.\\
			\end{align*}
			Let us check if 
			\begin{equation}
			(x)((h_{a}+h_{b})\circ h_{c})=(x)(h_{a}\circ h_{c})+(x)(h_{b}\circ h_{c})
			\end{equation}
			In fact, \begin{align*}
			(x)((h_{a}+h_{b})\circ h_{c})&=((x)(h_{a}+h_{b})h_{c}\\
			&=((x)h_{a}+(x)h_{b})h_{c}\\
			&=(a+b)h_{c}\\
			&=c.
			\end{align*}
			Also, \begin{align*}
			(x)(h_{a}\circ h_{c})+(x)(h_{b}\circ h_{c})&=((x)h_{a}h_{c})+((x)h_{b})h_{c}\\
			&=(a)h_{c}+(b)h_{c}\\
			&=c+c.
			\end{align*}
		\end{itemize}
		Since $c\neq 0$, $c\neq c+c$.
		From this claim, we conclude that not every near ring is a ring but every ring is necessarily  a near ring.
		Furthermore, let $y, y\in G$, $y\neq z$ and $f, g\in M(G)$ such that
		\begin{equation*}
		\begin{cases}
		(x)f_{y}:=y,\\
		(x)g_{z}:=z.
		\end{cases}
		\end{equation*}
		We have to check if $(x)((f_{y}+g_{z})\circ h)= (x)(f_{y}\circ h)+ (x)(g_{z}\circ h)$.
		In fact, 
		\begin{align*}
		(x)((f_{y}+g_{z})\circ h)&=((x)(f_{y}+g_{z})h\\
		&=((x)f_{y}+(x)g_{z})h\\
		&=(y+z)h.
		\end{align*}
		Also, 
		\begin{align*}
		(x)(f_{y}\circ h)+ (x)(g_{z}\circ h)&=((x)f_{y})h+((x)g_{z})h\\
		&=yh+zh.
		\end{align*}
		If $G$ contains more than one element, then the right distributive low does not hold in $M(G)$ and not all mappings  in $M(G)$ are endomorphism. For $(x)((f+ g)\circ h)= (x)(f\circ h)+ (x)(g\circ h)$ to hold, $h$ must be an endomorphism in this case.
		
		Then from this claim, every ring is near-ring but the inverse is not always true,  so we have the following.
	\end{proof}
\end{exa}
\begin{defn}
	A proper nearfield is a nearfield that is not a field.
\end{defn}
\begin{thm}(\textbf{Zassenhaus})
	The additive group of a nearfield  is  abelian.
\end{thm}
For the proof See \cite{pilz2011near}.

Note that the nearfield is a near-ring with identity
such that each non-zero element has an inverse.
Hence, each nearfield is a near-ring but the converse
is not true. For example, $(\mathbb{Z}, +, \cdot)$ where $\mathbb{Z}$ is the set of integers with usual addition ($+$) and usual ($\cdot$) multiplication is a near-ring but it is not a near-field. The symbols $0$ and $1$ will be used for the additive and multiplicative identities, respectively. In a near-ring or a nearfield $R$, $1\neq 0$ . If $1=0$, then for all $x$ we have $x=x1=x0=0$ , so $R=\left\lbrace 0\right\rbrace $
contradicting the assumption that $R$ has at least two
elements (\cite{hussein2022some} p102).

\begin{rem}
	All fields are near fields and also any
	division ring is a near field  (\cite{hussein2022some}).
\end{rem}
\begin{defn}
	A left near ring is said to be \emph{zero-symmetric} if
	$0n=0$, for all $n$ in $R$ , i.e., the left distributive law
	results in $n0=0$. The set of all zero-symmetric
	elements of $R$ is denoted by $R_{0}=\left\lbrace x\in R: 0x=0\right\rbrace $,  referred to the zero-symmetric part of $R$ . If $R=R_{0}$
	, then $R$ is  zero-symmetric. The zero-symmetric near-ring is sometimes named as C-ring (\cite{hussein2022some}).
\end{defn}
\begin{defn}(\cite{raovague})
	
	A \emph{division ring}, also called \emph{skewfield} is nontrivial ring in which division by nonzero elements is defined. Specifically, it is a nontrivial ring in which every nonzero  element $a$ has a multiplicative inverse denoted $a^{-1}$ such that $aa^{-1}=1$.
	
	So (right) division ring may be defined as $ \frac{a}{b}=ab^{-1}$  but this notation is avoided as one may have $ ab^{-1}\neq b^{-1}a. $
	
	Notice that any division ring is a nearfield and any commutative division ring is a field.
\end{defn}
\begin{defn}\label{sbnf}
	Consider a nearfield $ R $. A subset $ S $ of  $R $ is said to be a \emph{subnearfield} of  $ R $ if $(S, +)$ and $(S^{*}, \cdot)$ are both groups. If moreover $ (b+c)a=ba+ca , \forall a, b, c \in S$, then $ S $ is a subfield of $ R $ (\cite{groves1974locally}).
\end{defn}

\begin{defn}\label{vs}
	A set $V$ is said to be a \emph{(left) vector space over a nearfield  $R$}, if $(V , +)$ is an abelian
	group and, if for each $\alpha\in R$ and $v \in V$, there is a unique element $v \alpha \in V$.  Moreover  the	following conditions hold for all $\alpha, \beta \in R$ and  for all $u, v\in V$:
	\begin{itemize}
		\item[(i)] $\alpha(u+v)= \alpha u+\alpha v$;
		\item[(ii)] $(\alpha+\beta)v=\alpha v+\beta v$;
		\item[(iii)] $(\alpha\beta) v=\alpha(\beta v)$;
		\item[(iv)]  $1v=v.$
	\end{itemize}
	The members of $V$ are called vectors and the members of the division ring (nearfield) are called
	scalars. The operation that combines a scalar $\alpha$ and a vector $v$ to form the vector $\alpha v$ is
	called scalar multiplication (\cite{howell2007contributions}).
\end{defn}

\subsection{Center and kernel of a nearfield }
As we saw from the definition of nearfield we do not necessarily have the right distributive law and and commutativity of multiplication. For that reason, the following concept can be defined (\cite{djagba2020generalized}), (\cite{djagba2019contributions}).

\begin{defn}
	Let $ R $ be a nearfield. The \emph{multiplicative center} $ (R, \cdot) $   denoted by $C(R)$  is defined as follows:
	\begin{equation}
	C(R)=\left\lbrace x\in R: xy=yx, \text{$\forall y \in R$} \right\rbrace. 
	\end{equation}  
	In others words, it is the set of all elements of $R$ that commute with every element of $R$.
\end{defn}\label{def221}
Here we use $D(R)$ to express the set of all distributive elements of a nearfield  $R$. It is  defined as follow (\cite{djagba2020generalized}):

\begin{equation}\label{der}
D(R)=\left\lbrace  x\in R : (y+z)x= yx+ zx, \text{ for all $ y, z \in R$ }\right\rbrace. 
\end{equation}
\begin{rem}\label{rem}
	We can see simply that $C(R)\subset D(R)$ (\cite{pilz2011near}). 
	
	Let $\alpha$ be in $C(R)$ and $\beta, \gamma\in R.$
	By definition, we know that $\alpha\in C(R)$ implies that for all $\beta, \gamma\in R$, we have \begin{align*}
	\alpha(\beta+\gamma)&=\alpha\beta+\alpha\gamma\\
	&=\alpha\beta+\alpha\gamma\\
	&=\beta\alpha+\gamma\alpha\\
	&=(\beta+\gamma)\alpha.
	\end{align*}
	That means, $\alpha(\beta+\gamma)=(\beta+\gamma)\alpha$.
	
	Then $(\beta+\gamma)\alpha=\beta\alpha+\gamma\alpha$. Thus $\alpha\in D(R)$.
	
\end{rem}
The remark (\ref{rem}) show that $C(R)\subset D(R)$ for the usual multiplication; but it is not direct to say that $D(R)\subset C(R)$. We will see it in the next chapter for a new multiplication which was introduced by Dickson but the notion of center of a nearfield is more developed in (\cite{djagba2020center}). First, we have the next theorem to show the relationship between $D(R)$ and $R$ where $R$ is a left nearfield.
\begin{thm}\label{theo2}
	Let $R$ be a near-field (\cite{howell2007contributions}). Then
	\begin{itemize}
		\item[(a)] $D(R)$  with operation of $R$ is a skewfield (division ring), and
		\item[(b)] $R$  is a left vector space over $D(R)$.
	\end{itemize}
\end{thm}
\begin{proof}
	
	\begin{itemize}
		\item[(a)] From the definition (\ref{sbnf}), we have to show that $(D(R), +)$ and $(D(R)^{*}, \cdot)$ are both groups.
		
		\begin{itemize}
			\item[(i)] Since $1\in D(R)$, $D(R)\neq \emptyset$ and $D(R)\subseteq R$.
			In fact, $(\alpha+\beta)\cdot 1=\alpha\cdot 1+\beta\cdot 1=\alpha+\beta$ for all $\alpha, \beta \in R$.
			\item[(ii)] Let $x$ and $y$ be elements of $D(R)$ and $\alpha, \beta$ elements of $R$. 
			We say that $x+y$ belong to $D(R)$ if $(\alpha+\beta)(x+y)=\alpha(x+y)+\beta(x+y)$ for all $\alpha, \beta \in R$.
			Then, 
			\begin{align*}
			(\alpha+\beta)(x+y)&=\gamma(x+y) \quad \text{ where $  (\alpha+\beta)=\gamma $}\\
			&=\gamma x+\gamma y\quad\text{(by definition of righr nearfield)}\\
			&=(\alpha+\beta)x+(\alpha+\beta) y\\
			&=\alpha x+\beta x+\alpha y+\beta y\quad \text{because $x, y \in D(R)$}\\
			&=\alpha x+\alpha y+\beta x+\beta y \quad \text{(because the additive group of a nf is abelain)}\\
			&=(\alpha x+\alpha y)+(\beta x+\beta y)\quad \text{because $'+'$ is associative in $R$}\\
			&=\alpha(x+ y)+\beta(x+y).
			\end{align*} 
			Hence $x+y \in D(R)$ and $(D(R), +)$ is a subgroup of $(R, +)$.
			\item[(iii)] $(D(R)^{*}, \cdot)$ is a subgroup of $(R^{*}, \cdot)$ 
			
			Let  $x\in D(R)^{*}$ and consider $x^{-1}$, then
			\begin{align*}
			\left[ (\alpha+\beta)x^{-1} \right] x&=(\alpha x^{-1}+\beta x^{-1} )x\\
			&=(\alpha x^{-1})a+(\beta x^{-1} ) x\\
			\end{align*}
			\begin{align*}
			\text{Which implies that} \left[ (\alpha+\beta)x^{-1} \right] x-(\alpha x^{-1})x+(\beta x^{-1} ) x&=0\\
			x\left[(\alpha+\beta) x^{-1}-(\alpha x^{-1}+\beta x^{-1}) \right] &=0\\
			x\neq 0 \Rightarrow (\alpha+\beta) x^{-1}-(\alpha x^{-1}+\beta x^{-1})&=0\\
			\Rightarrow (\alpha+\beta) x^{-1}&= \alpha x^{-1}+\beta x^{-1},
			\end{align*}
			\text{for all $\alpha, \beta \in R$}.
			So $x^{-1}\in D(R)$.
			\item[(iv)] $D(R)$ is closed under multiplication.
			Let $x, y \in D(R)$. Then we want to show that $(\alpha+\beta)(xy)=\alpha xy+\beta xy$ for all $\alpha, \beta \in R.$
			
			In fact, 
			\begin{align*}
			(\alpha+\beta)xy&=\left[(\alpha+\beta)x \right] y\\
			&=\left[ (\alpha x)+ (\beta x)\right] y \quad\text{by deinition of right neafield} \\
			&=(\alpha x)y+ (\beta x)y\\
			&=\alpha xy+\beta xy\quad \text{because $x, y \in D(R)$}.
			\end{align*}
			Hence $xy \in D(R)$. Therefore $ (D(R)^{*}, \cdot) $ is a subgroup of $R$.
			And then $D(R)$ is a subnearfield of $R$
			\item[(v)] For all $x, y, z\in D(R)$, $(x+y)z=xz+yz$ and $x(y+z)=xy+xz$ are satisfied (they are shown in the previous steps).
			We conclude that $D(R)$ is a skewfield (division ring).
		\end{itemize}
		\item[(b)] Using the definition (\ref{vs}), we have:
		\begin{itemize}
			\item[(i)] $(R, +)$ is an abelian group.
			\item[(ii)] $\forall x\in D(R)$ and $\alpha \in R$, Let $\alpha=\alpha_{1}, \dots, \alpha_{n}$. Then
			\begin{align*}
			x(\alpha_{1}, \dots, \alpha_{n})&=(x\alpha_{1}, \dots, x\alpha_{n})\\
			&=x\alpha\in R. 
			\end{align*}
			\item[(iii)] $\forall x\in D(R)$ and $\alpha, \beta \in R$, 
			\begin{align*}
			x(\alpha_{1}, \dots, \alpha_{n}+\beta_{1}, \dots, \beta_{n})&=x\left[(\alpha_{1}+\beta_{1}),\dots, (\alpha_{n}+\beta_{n})\right] \\
			&=\left[x(\alpha_{1}+\beta_{1}),\dots, x(\alpha_{n}+\beta_{n})\right] \\
			&=x(\alpha_{1}, \dots, \alpha_{n})+x(\beta_{1}, \dots, \beta_{n})\\
			&=x\alpha+ x\beta.
			\end{align*}
			\item[(iv)] For all $x, y \in D(R)$ and for all $\alpha \in R$,
			\begin{align*}
			(x+y)(\alpha_{1}, \dots, \alpha_{n})&=x(\alpha_{1}, \dots, \alpha_{n})+y(\alpha_{1}, \dots, \alpha_{n})\\
			&=x\alpha+ y\alpha.
			\end{align*}
			\item[(v)] For all $xy \in D(R)$ and $\beta \in R$, $(xy\beta)=x(y\beta)$.
			Then $R$ is a vector space over $D(R)$.
		\end{itemize}
	\end{itemize}
\end{proof}
It is clear that every nearfield is a near ring. So $D(R)$ is a subnear ring in case that $R$ is a near ring. We have shown in \ref{lemera} that the set of mappings $M(G)$ is a near ring where only the left distributive law holds.In that case we define a new set $D(M(G))$ of all distributive maps where the right distributive low holds. The set of all distributive maps for all $x\in G$ is defined and denoted as follow.

\begin{equation}
D(M(G)):=\left\lbrace h\in M(G): (x)((f+g)\circ h) =(x)(f\circ h)+ (x)(g\circ h), \quad\text{for all}\quad f, g\in M(G)\right\rbrace. 
\end{equation}
We have shown that $M(G)$ is a near ring. In order we can show that $D(M(G))$ is a subnear ring of $M(G)$.

\textbf{Claim}: \emph{If $M(G)$ is a near ring, then $D(M(G))$ is a subnear ring of $M(G)$.} 

\begin{proof}
	\begin{itemize}
		\item[$\bullet$] Since $M(G)\neq \emptyset$, $D(M(G))\neq \emptyset$.
		\item[$\bullet$] Let $k_{1}, k_{2}\in D(M(G))$. Then for all $f, g\in M(G)$ we have
		\begin{align*}
		(x)\left[(f+g)\circ (k_{1}+k_{2}) \right]&=(x)\left[\alpha\circ(k_{1}+k_{2}) \right]\quad\text{ (we let $\alpha= (f+g)$)}\\
		&=((x))(k_{1}+k_{2})\quad\text{(by definition of $"\circ"$ )} \\
		&=((x)\alpha)k_{1} +((x)\alpha)k_{2}\quad\text{(by definition of $"\circ"$ )} \\
		&=((x)(f+g))k_{1}+((x)(f+g))k_{2}\\
		&=((x)f)k_{1}+((x)g)k_{1}+((x)f)k_{2}+((x)g)k_{2}\quad \text{( $k_{1}, k_{2} \in D(M(G))$)}\\
		&=((x)f)k_{1}+((x)f)k_{2}+((x)g)k_{1}+((x)g)k_{2}\quad\text{(M(G), +) is abelian}\\
		&=\left[((x)f)k_{1} +((x)f)k_{2}\right]+ \left[((x)g)k_{1} +((x)g)k_{2}\right]\quad\text{+ is associative}\\
		&=((x)f)(k_{1}+k_{2})+((x)g)(k_{1}+k_{2})\\
		&=(x)(f\circ (k_{1}+k_{2}))+(x)(g\circ (k_{1}+k_{2})).
		\end{align*}
		Hence $k_{1}+k_{2}\in D(M(G))$. Therefore $(D(M(G)), +)$ is a subgroup of $(M(G), +)$. Since  $(M(G), +)$, $(D(M(G)), +)$ is also abelian.
		\item[$\bullet$] Let $k_{1}, k_{2}\in D(M(G))$. If we are able to show that $(x)\left[ (f+g)\circ (k_{1}\circ k_{2})\right]=(x)(f\circ (k_{1}\circ k_{2}))+(x)(g\circ (k_{1}\circ k_{2}))$, we conclude that $D(M(G))$  is closed under multiplication and the associativity is verified.
		In fact, 
		\begin{align*}
		(x)\left[ (f+g)\circ (k_{1}\circ k_{2})\right]&=((x)(f+g))(k_{1}\circ k_{2})\\
		&=((x)f+(x)g)(k_{1}\circ k_{2})\\
		&=(((x)f+(x)g)k_{1})k_{2}\quad\text{(by definition of $\circ$)}\\
		&=(((x)f)k_{1}+((x)g)k_{1})k_{2}\\
		&=\left((x)(f\circ k_{1})+(x)(g\circ k_{1}) \right)k_{2}\\
		&=\left( (x)(f\circ k_{1})\right) k_{2}+\left((x)(g\circ k_{1} \right)k_{2}\\
		&=((x)f)(k_{1})k_{2}+  ((x)g)(k_{1})k_{2}\\
		&=(x)(f\circ (k_{1}\circ k_{2} ))+ (x) (g\circ (k_{1}\circ k_{2} )).
		\end{align*}
		Hence $(x)(k_{1}\circ k_{2})\in D(M(G))$. Therefore $(D(M(G)), \circ)$ is a sub-semi-group of $(M(G), \circ)$. 
		\item[$\bullet$] For all $(x)k\in (D(M(G))^{*}, \circ )$, there exists $(x)k^{-1}$ such that $(x)(k\circ k^{-1})= x$.
		So, 
		\begin{align*}
		(x)\left\lbrace \left[(f+g) \circ k^{-1}\right]\circ k\right\rbrace&=(x)\left[((f+g)k^{-1})\circ k \right]  \\
		&=\left(((x)f+ (x)g )k^{-1}\right) k\\
		&=\left(((x)(f))k^{-1} +((x)(g))k^{-1}\right)k\\
		&= ((x)(f\circ k^{-1})+(x)(g\circ k^{-1}))k.  
		\end{align*}
		Since $(x)k\neq 0$, we have 
		\begin{align*}
		(x)\left\lbrace \left[(f+g) \circ k^{-1}\right]\circ k\right\rbrace-&((x)(f\circ k^{-1})+(x)(g\circ k^{-1}))k=0\\
		\Leftrightarrow((x) ((f+g) \circ k^{-1}))k&=((x)(f\circ k^{-1})+(x)(g\circ k^{-1}))k\\
		\Leftrightarrow (x)((f+g) \circ k^{-1})&=(x)(f\circ k^{-1})+(x)(g\circ k^{-1}).
		\end{align*}
		Then $(x)k^{-1}\in D(M(G))$.Hence $(D(M(G))^{*},\circ )$ is a subgroup of $(M(G)^{*}, \circ)$. Therefore $(D(M(G)), +,\circ )$ is a sub-near ring of $M(G)$.
	\end{itemize}
\end{proof}
\begin{lem}
	Every nearfield $  R$ contains  a  commutative subfield $ \mathbb{F} $ (there is (possibly) different sub-near-fields in  $  R$)(\cite{pilz2011near}).
\end{lem}
We need to show that there exists a subnearfield  $ \mathbb{F} $ of $ R $ that satisfies the two following conditions.
\begin{enumerate}
	\item	We have  to  show that $ (\mathbb{F}, +) $  is a group
	\item	We have  to  show that $ (\mathbb{F}^{*}, \cdot) $  is a group (\cite{pilz2011near}).
\end{enumerate}

A  near-ring can be left or right depending on the author. There exist relationship between the additive group $M$ and the near-ring $R$. Here it is about left nearing. The we have the following.
\begin{defn}
	An additive group $(M, +)$ is called a \emph{(right) rear ring module over a (left) near ring} $R$  if there exist a mapping 
	\begin{align*}
	\phi: &M\times R\rightarrow M\\
	&(m, r)\rightarrow mr
	\end{align*} such that $m(r_{1}+r_{2})=mr_{1}+mr_{2}$ and $m(r_{1}r_{2})=(mr_{1})r_{2}$ for all $r_{1}, r_{2}\in R$ and $m\in M.$
\end{defn}

We write $M_{R}$ to denote that $M$ is a right near ring module over a left near ring $R.$

\begin{defn}
	A subset $A$  of a near ring module $M_{R}$ is called a \emph{ $R-$subgroup} if $A$ is a subgroup of $(M, +)$, and $AR=\left\lbrace ar: a\in A, r\in R \right\rbrace \subseteq A$.
\end{defn}
\begin{defn}
	A subset $H$ of a near-ring module $ M_{R} $ is called an \emph{$ R $-subgroup}
	if:
	\begin{itemize}
		\item[(i)] $ H $ is a subgroup of $ (M,+) $,
		\item[(ii)] $ HR = \left\lbrace hr : h\in  H, r\in R  \right\rbrace \subseteq H. $
	\end{itemize}
	For any left near-ring $R$ is we can construct an $R$-subgroup respect to some conditions. Therefore, we have the following. 
	\begin{defn}
		A nearring module $ M_{R} $ is said to be \emph{irreducible} if $ M_{R} $
		contains no proper $ R $-subgroups. In other words, the only $ R $-subgroups of $ M_{R} $ are
		$ M_{R} $ and $\left\lbrace 0\right\rbrace $.
	\end{defn}
\end{defn}
\begin{defn}
	A nearring module $ M_{R } $ is called \emph{strictly semi-simple} if $ M_{R } $ is a direct sum of irreducible submodules (\cite{djagba2020generalized}).
\end{defn}
So we have, the following definition.

%\begin{defn}
%Let $R$ be a nearfield and $ M_{R } $  be a nearring module. $ M_{R } $
%is called a \emph{Beidleman near-vector space} if $ M_{R } $ is a strictly semi-simple nearring
%module (\cite{djagba2020generalized}).
%\end{defn}
%Let $R$ be a nearfield and $ M_{R} $ an irreducible nearring module.
%Then $ M_{R} $ is a Beidleman near-vector space. Moreover,
%$ M_{R}\cong R_{R}$.
%As with vector spaces, we have the same notion of "linear combination" (\cite{djagba2020generalized}).
\begin{defn}\cite{howell2007contributions}
	Suppose $ M_{R } $ and $ N_{R } $ are nearring modules. The map $\phi$
	from $ M_{R } $ into $ N_{R } $ is called an \emph{$ R $-homomorphism} if $\phi (xr)=\phi(x)r$ and $\phi(x+y)=\phi(x)+\phi(y)$ for all $x, y\in M$ and $r\in R$.
	
	If $\phi$ is bijective then, $\phi$ is called an \emph{$R$-isomorphism}.
\end{defn}
An epimorphism is a surjective homomorphism and a monomorphism is an injective
homomorphism. If a homomorphism is bijective, i.e. surjective and injective,
it is called an isomorphism. A homomorphism g from a set to itself is called an
endomorphism. If g is bijective, it is called an automorphism.

We say that $M_{R}$ is
embedded in $N_{R}$ if there exists a monomorphism from $M_{R}$ to $N_{R}$.
The set of all nearring homomorphisms from $ M_{R} $ to $N_{R}$ is denoted by $Hom(M_{R}, N_{R})$ (\cite{howell2007contributions}).

As we said at the beginning of this chapter, we did not give all details of the concepts that we need, but we gave the basic elements on a nearfield theory and some results on nearfileds an near-rings. Because we are studying the generalized set of all distributive elements of a near field,  we will see in next chapter how to construct a Dickson nearfield.

\newpage
\section{Construction of finite Dickson Nearfield}

In this chapter we are going to define a new multiplication and present the construction of a finite Dickson  nearfield. First we define Dickson  pair.
Dickson obtained the first proper nearfields in $1905$ by distoring the multiplication in a finite field. 
\subsection{Dickson Nearfield}
A Dickson nearfield is ''twisting'' of a field where we define the twisting by a Dickson pair (\cite{boykett2016multiplicative}). Now we have the following definition.
\begin{defn}
	A pair  of positive integers $(q, n)$ is said to be a \emph{Dickson pair} if the following conditions are satisfied:
	\begin{itemize}
		\item[(i)] $q$ is of the form$p^{l}$ for  some prime  $p$;
		\item[(ii)] each prime divisor of $n$ divides $q-1$;
		\item[(iii)] $ q \equiv 3 \mod 4$ implies $4$ does not divide $n$ (\cite{djagba2020generalized}).
	\end{itemize}	
\end{defn}
\begin{exa}
	$(7, 9), (4, 3), (5, 4), (19, 6)$ are all Dickson pairs.
\end{exa}
Let $ (q, n) $ be a Dickson pair and $k\in \left\lbrace 1, \dots, n \right\rbrace$; 
we denote the positive integer $\frac{q^{k}-1}{q-1}$ by $[k]_{q}$.

\begin{rem}
	\begin{itemize}
		
		\item[(i)] Let $(q, n)$ be a Dickson pair. Then $n$ divides $[n]_{q}$.
		\item[(ii)] Since every prime divisor of $n$ divides $q-1$, then $\gcd(q, n)=1$.
		\item[(iii)] $x_{n-1}\equiv [n]_{q}\mod n$ satisfies the recurrence 
		\begin{align*}
		x_{n}&\equiv q x_{n-1}+1 \mod n\\
		\Leftrightarrow 1&\equiv qx_{n-1}+1\mod n\\
		\Leftrightarrow qx_{n-1}&\equiv 0\mod n.
		\end{align*}
	\end{itemize}
	From $(iii)$, we must have that 
	\begin{align*}
	\Leftrightarrow x_{n-1}&\equiv 0\mod n.
	\end{align*}
\end{rem}
Also note that all Dickson nearfields arise by taking Dickson pair as discribed in the Theorem 8.31 (\cite{pilz2011near}. p244 ). In this thesis, the
set of  Dickson nearfields  for any Dickson pair $(q, n)$ is denoted by $DN(q, n)$, and the Dickson nearfield arising from the Dickson pair $(q, n)$ with a generator $g$ for the finite field of order $q^{n}$ is denoted by  $ DN_{g}(q,n) $. The multiplicative group arising by a Dickson pair $(q, n)$ is denoted by $G_{q, n}$. The group $G_{q, n}$ is metacyclic and can be presented as follow
\begin{equation}
\left\langle a, b \mid a^{m}=1, b^{m}=a^{t}, ba= a^{q}b\right\rangle 
\end{equation} which  is the set of the elements $a, b$, 
where \begin{equation}
\begin{cases}
m=\frac{q^{n}-1}{n}\\
t=\frac{m}{q-1}.
\end{cases}
\end{equation}
Now to construct a finite Dickson nearfield, we need the following.
\begin{defn}(\cite{pilz2011near})
	Let $R$ be a nearfield and $ Aut(R, +, \cdot) $ the set of all automorphism of $R$. A map 
	\begin{align*}
	\phi: R^{*}&\longrightarrow  Aut(R, +, \cdot) \\
	a&\longrightarrow \phi_{a}
	\end{align*} is called a coupling map if for all $a,b \in R^{*}$ we have $\phi_{a}\circ\phi_{b}=\phi_{\phi_{a}(b)\cdot a} $.
\end{defn}
\begin{exa}(\cite{pilz2011near})
	Let us consider 
	\begin{align*}
	\phi: R^{*}&\longrightarrow  Aut(R, +, \cdot) \\\quad %\text{defined by} a\longrightarrow \phi_{a}\\
	a&\longrightarrow id_{R}.
	\end{align*}
	The map $\phi$ is a coupling map because for all $a, b \in R^{*}$, we have $\phi_{a}\circ \phi_{b}=id_{R}\circ id_{R}= id_{R}$ and $\phi_{a}(b)= b$. Then 
	\begin{align*}
	\phi_{\phi_{a}(b)}a &=\phi_{ba}\\
	&=id_{R}.\\ 
	\text{Therefore}\quad \phi_{a}\circ \phi_{b}&=\phi_{\phi_{a}(b)}a.
	\end{align*}
\end{exa}
\subsection{Dickson construction}
To define a Dickson nearfield, Dickson used  a technique to "distort" the multiplication of a finite field. 
\begin{defn}(\cite{djagba2020generalized})
	Let $(R, +, \cdot)$ be a nearfield. Let us consider the coupling map $\phi: a\longrightarrow id$ . In this case 
	\begin{equation*}
	a\circ_{\phi} b:=
	\begin{cases}
	\phi_{a}(b)\cdot a=a\cdot id (b)= a\cdot b, \quad \text{if}\quad a\neq 0,\\
	0,\qquad \text{if}\quad a=0
	\end{cases}
	\end{equation*}
\end{defn}
Thus we have the trivial coupling map because the new operation is the same as the usual operation of multiplication. And then we have the following definition.    
\begin{defn}(\cite{pilz2011near})
	If $(R, +, \cdot)$ is a nearfield, then the \emph{$\phi-$derivation} of   $(R, +, \cdot)$ is  $(R, +, \circ_{\phi})$ which means $R^{\phi}= R$ is also a nearfield but not necessarily a Dickson nearfield.
\end{defn}
\begin{exa}(\cite{djagba2020generalized})
	Let $(\mathbb{H}, +, \cdot)$ be skewfield of real quaternions (with the standard basis $\left\lbrace 1, i, j, k \right\rbrace $) and $t\in R$. We define a \emph{new multiplication}  $"\circ"$  on $\mathbb{H}$ by 
	\begin{equation*}
	a\circ b=\begin{cases}
	\mid b\mid^{it}a \mid b\mid^{-it} b \quad \text{if} \quad b\neq 0\\
	0 \quad \text{if} \quad b= 0\\
	\end{cases}
	\end{equation*}
	Then $(\mathbb{H}_{t}:= (\mathbb{H}, +, \circ))$ is a nearfield but not a Dickson nearfield.
	In fact  $\mathbb{H}_{t}=\mathbb{H}^{\phi}$ where
	\begin{align*}
	\phi: \mathbb{H}&\longrightarrow Aut(\mathbb{H}, +, \cdot)\\
	b&\longrightarrow \phi_{b} 
	\end{align*} is a coupling map with automorphism 
	\begin{align*}
	\phi_{b}: & \mathbb{H}\longrightarrow \mathbb{H}\\
	&a\longrightarrow \mid b\mid^{it} a \mid b\mid^{-it}.
	\end{align*}
\end{exa}
\begin{defn}(\cite{boykett2016multiplicative})
	If $(\mathbb{F}, +, \cdot)$ is a field, then the $\phi-$ derivation of $(F, +, \cdot)$ is $(\mathbb{F}, +, \circ_{\phi})$ which means $\mathbb{F}^{\phi}=\mathbb{F}$. It implies that every field is a Dickson nearfield.
\end{defn}
\begin{defn}(\cite{pilz2011near})
	The notation  $ R^{\phi}$. $\left\lbrace \phi_{a}: a\in R^{*}\right\rbrace $ is called the Dickson-group of $\phi$. $ R$
	is said to be a Dickson nearfiled if  $ R$ is the $\phi-$derivation of some field $\mathbb{F}^{\phi}$,  $(\mathbb{F}^{\phi}= R)$.
\end{defn}
We will see that for each Dickson pair $(q, n)$, a Dickson  nearfield contains $q^{n}$ elements.

\begin{thm}
	For all Dickson pairs $(q, n)$,  there exists some associated finite nearfield of order  $q^{n}$  which arise by taking   the finite field $\mathbb{F}_{q^{n}}$ and change the multiplication such that  $\mathbb{F}^{\phi}_{q^{n}} =(\mathbb{F}_{q^{n}}, +, \circ)$ for  some coupling map $\phi$ on $\mathbb{F}_{q^{n}}$,  where $"\circ"$ is the new multiplication (\cite{pilz2011near}).
\end{thm} 

For the proof, see (\cite{pilz2011near}).

\begin{thm}\label{thg}
	Let $R$ be a finite Dickson nearfield that arises from the Dickson
	pair $(q, n)$. Then $D(R) = \mathbb{F}_{q}$ (\cite{djagba2020generalized}).
\end{thm}

\begin{proof}
	Let  $(q, n)$  be a Dickson pair where $q=p^{l}$ for some prime $p$ and positive integers $l, n$. Let us consider  $g$ a generator of $\mathbb{F}_{q^{n}}^{*}$ and $R$ the finite nearfield which is constructed with $H=<g^{n}>$. Let $\mathbb{F}_{q}$ be the unique subfield of order $q$ of $\mathbb{F}_{q^{n}}$. Then $\mathbb{F}_{q}\subseteq D(R)$.

	From a lemma in (\cite{murphy2006course}), we know that $\mathbb{F}_{q}$ is a solution set to the equation $\alpha^{q}-\alpha=0$ in $\mathbb{F}_{q^{n}}$.
	We consider $g$ a generator of  $\mathbb{F}_{q^{n}}^{*}$ and we take  $\alpha\in \mathbb{F}_{q}^{*}$ and we write $\alpha=g^{l}$.  Since $\alpha\in \mathbb{F}_{q}$, we have $\alpha^{q}=\alpha$; which means $\alpha^{q-1}=1$. Therefore $(g^{l})^{q-1}=1$, which means $g^{(q-1)}=1$. 
	
	Thus, $q^{n}-1$ divides $l(q-1)$, i.e $[n]_{q}l$.
	Thus, $\mathbb{F}^{*}_{q}=<g^{[n]_{q}}>$. Since $n$ divides $[n]_{q}$, then $<g^{[n]_{q}}>$ is a subset of $<g^{n}>$. Then  we have $\mathbb{F}^{*}_{q}\subseteq H$.
	
	Furthermore, for $\alpha \in\mathbb{F}^{*}_{q} $, $x\in H=g^{[n]_{q}}H$.
	By Dickson construction, 
	\begin{align*}
	\phi_{\alpha}(\beta)&=\varphi^{n}(\beta)\\
	&=\beta^{q^{n}}\\
	&=\beta.
	\end{align*}
	Hence $ \phi_{\alpha}= id$.
	We  take
	\begin{align}
	(y+z)\circ \phi_{\alpha}(t)&=y\cdot \phi_{\alpha}(t)+z\cdot \phi_{\alpha}(t)\label{eq}\\
	&=yt+zt\quad\text{for all $y, z, t \in R$}\nonumber.
	\end{align}
	Moreover, since $\alpha \in \mathbb{F}_{q}$, then $\alpha^{q}=\alpha$.
	Thus $\varphi^{l}(\alpha)=\alpha$ and
	\begin{align*}
	(y+z)\circ\alpha&=(y+z)\phi_{(y+z)}(\alpha)\\
	&=(y+z)\cdot \varphi^{l}(\alpha)\\
	&=y\varphi^{l}(\alpha)+z\varphi^{l}(\alpha)\quad\text{ from (\ref{eq})}\\
	&=y\alpha+z\alpha\quad\text{because $\varphi^{l}(\alpha)=\alpha$}.
	\end{align*}
	Therefore for all $y, z, t \in R$, $\alpha \in D(R)$. It is proved that $\mathbb{F}_{q}\subseteq D(R).$
	
	Let us how that $D(R)\subseteq \mathbb{F}_{q}$.
	
	Let 	$\alpha\in D(R)$ therefore $(y+z)\circ\alpha=y\circ\alpha+z\circ \alpha$, for all $y, z\in R$.
	Let $(y+z)=g^{n}H$. Then 
	\begin{align*}
	(y+z)\circ \alpha&=(y+z)\phi_{(y+z)}(\alpha)\\
	&=g^{n}\phi_{g^{n}}(\alpha)\\
	&=g^{n}\alpha\quad\text{since $\phi_{g^{n}}=id$}.
	\end{align*}
	Since $(y+z)\alpha=y\circ\alpha+z\circ \alpha$, 
	\begin{align*}
	(y+z)\phi_{(y+z)}(\alpha)&=y\phi_{y+z}(\alpha)+z\phi_{y+z}(\alpha)\\
	&=y\phi_{g^{n}}(\alpha)+z\phi_{g^{n}}(\alpha),
	\end{align*}
	and $g^{n}\alpha= \alpha\phi_{\alpha}(g^{n})$. Hence $\phi_{\alpha}(g^{n})= g^{n}$.
	
	Furthermore, since $\mathbb{F}_{p}$ is fixed by $\psi$ , the Frobenius map, $\phi_{\alpha}$ fixes $\mathbb{F}_{p}$. Therefore $\phi_{\alpha}$ fixes $\mathbb{F}_{p}(g^{n})$, the smallest subfield of $\mathbb{F}_{q^{n}}$ that contains $\mathbb{F}_{p}$ and $g^{n}$. By the lemma 2.8 (\cite{djagba2020generalized}), $\phi_{\alpha}$ fixes $\mathbb{F}_{q^{n}}$ . Thus $\phi_{\alpha}=id$.
	
	Let us take $(y+z)=g\in g^{[n]_{q}}H$, then $\phi_{(y+z)}=\phi_{g}=\varphi=\psi^{l}$. So
	
	\begin{align*}
	(y+z)\circ\alpha&=g\circ x\\
	&=g\phi_{g}(\alpha)\\
	&=g\varphi(\alpha).
	\end{align*}
	
	We have now 
	\begin{align*}
	(y+z)\circ \alpha&=y\circ \alpha+z\circ \alpha\quad (\text{because $\alpha \in D(R)$})\\
	\Leftrightarrow g\varphi(\alpha)&=g\alpha\\
	\Leftrightarrow \varphi(\alpha)&=\alpha\\
	\Leftrightarrow \alpha^{q}&=\alpha.
	\end{align*} Therefore $\alpha \in \mathbb{F}_{q}$.
	We have shown that $D(R)=\mathbb{F}_{q}$ where $R\in DN(q, n)$.
\end{proof}

In this chapter, we defined the concept of a Dickson pair and showed examples of Dickson pairs. We then introduced the notion of a Dickson nearfield, which arises from a twisting of a finite field using a Dickson pair. We presented the construction of a finite Dickson nearfield using a coupling map and defined the new multiplication operation on a nearfield as the composition of the usual multiplication and the automorphism induced by the coupling map. Finally, we gave the presentation of the multiplicative group of a finite Dickson nearfield.

\newpage
\section{The generalized set of distributive elements of a nearfield}
In the first chapter, from the Definition \ref{der}, we have seen that if $R$ is a left nearfield, $D(R)$ is the set af all distributive elements of $R$. In  this chapter, we are going to study the generalized set of distributive elements of a nearfield $D(\alpha, \beta)$. We will see some sufficient conditions on $\alpha$ and $\beta$ for $D(\alpha, \beta)$ to be  a subfield of $\mathbb{F}_{q^{n}}$,  where $\mathbb{F}_{q^{n}}$ is a finite field of order $q^{n}$.

\subsection{New multiplication of Dickson }
Considering  a Dickson pair and let  $R$  be a finite Dickson near-field.
For a given pair $(\alpha, \beta)\in R^{2}$, we consider the set
\begin{equation}\label{ey}
D(\alpha, \beta)=\left\lbrace \lambda \in R: (\alpha+\beta)\circ
\lambda=\alpha\circ \lambda+\beta\circ\lambda \right\rbrace.
\end{equation}

In the equation  (\ref{ey}), '$\circ$' is the  new multiplication of the Dickson nearfield (\cite{djagba2020generalized}).

Note that the set $D(\alpha,\beta)$ is not always a subfield of $\mathbb{F}_{q^n}$   and it is not always  a subnearfield of $R$. There are some conditions on $\alpha$ and $\beta$ that can make $D(\alpha,\beta)$ a subfield of $\mathbb{F}_{q^n}$. Moreover, if $\alpha, \beta,$ and $\alpha+\beta$ belong to different sets, we can create a subfield of $\mathbb{F}_{q^n}$ using $D(\alpha,\beta)$ .
As we know, if $R$ is a left nearfield, then $D(R)$ is the set of all distributive elements  of $R$ and $C(R)$ is the center of $R$. Here $R$ is provided with the operations $'+'$ and $"\circ"$.
From this, we have the definition below.
\begin{defn}
	Let $R$ be a near-ring and $D(R)$ be the distributive elements of $R$. Then the \emph{generalized center of $R$} is defined as 
	\begin{equation}
	GC(R)=\left\lbrace x\in R: x\circ y= y\circ x, \text{for all}\quad y\in D(R)\right\rbrace.
	\end{equation}
\end{defn}
Given  $k\in \left\lbrace 1, \dots, n \right\rbrace $, an $H$-cosets is a coset of the form $g^{[k]_{q}}H$.
For any pair  $(\alpha,\beta)\in R^{2}$, we are going to see some conditions on  $\alpha, \beta, \alpha+\beta$ when  they  belong to the same $H$-cosets or when they are in  the different $H$-cosets.
This leads us to investigate the results below.
\subsection{Some results on $D(\alpha, \beta)$}
We just give some lemma and theorem in which we find some conditions on $\alpha$ and $\beta$ so that we can have a decision on $D(\alpha, \beta)$ over a finite field $\mathbb{F}_{q^{n}}$ for a given Dickson pair $(q, n)$. The first result belongs exactly on the definition of $D(\alpha, \beta)$ with a new multiplication. So we have the lemma bellow.
\begin{lem}\label{lemm31}
	Let $R\in DN(q, n)$ where $(q,n)$ is a Dickson pair. Let $(\alpha,\beta)\in R^{2}$. If $\alpha, \beta, \alpha+\beta$ belong to the same $H$-cosets, then  $(\alpha+\beta)\circ \lambda=\alpha\circ \lambda +\beta\circ\lambda$ for all $\lambda \in R$ (\cite{djagba2020generalized}). 
\end{lem}
\begin{proof}
	We consider $g$  such that 
	\begin{equation*}
	\begin{cases}
	\mathbb{F}_{q^{n}}=<g>,\\
	H=<g^{n}>.
	\end{cases}
	\end{equation*} 
	The set of all $H$-cosets is constructed as 
	\begin{equation}
	\mathbb{F}^{*}_{q^{n}}/H=\left\lbrace  H, g^{[1]_{q}}H, \dots,  g^{[n]_{q}}H\right\rbrace 
	\end{equation}
	Now we assume that $\alpha, \beta, \alpha+\beta\in  g^{[k]_{q}}H$ for $1\leq k\leq n$.
	We know that any finite field $\mathbb{F}_{q}$ of order $q$ is a set of solutions of the equation $X^{q}-X=0$ (\cite{lidl1997finite} p52). Then, 
	\begin{align*}
	(\alpha+\beta)\circ\lambda&= (\alpha+\beta)\circ \lambda^{q^{k}}\\
	&=\alpha\circ \lambda^{q^{k}}+\beta\circ \lambda^{q^{k}}\\
	&=\alpha \circ \lambda+ \beta\circ \lambda, \quad\text{for all}\quad \lambda\in R.
	\end{align*}
\end{proof}
\begin{lem}\label{lem2}
	Let $(q,n)=(p^{l}, 2)$ where $p$ is prime and $R\in DN(q, 2)$. Let $(\alpha, \beta)\in R^{2}$ and we assume that $\alpha, \beta,  (\alpha+ \beta)$ do not belong to the same $H$-cosets. We have that  $(\alpha+\beta)\circ
	\lambda=\alpha\circ \lambda+\beta\circ\lambda$ if and only if $\lambda\in D(R)$ (\cite{djagba2020generalized}).
\end{lem}
\begin{proof}
	Considering  that  $\alpha, \beta,  (\alpha+ \beta)$ are not all square or not all non square ( we  suppose that  $\alpha, \beta,  (\alpha+ \beta)$ belong  to different $H$-cosets). 
	
	Now we consider the case where  $\alpha+\beta\in H$ and $\alpha, \beta \in gH$. If  $(\alpha+\beta)\circ\lambda= \alpha\circ\lambda+\beta\circ\lambda$ then $(\alpha+\beta)\circ\lambda= \alpha\circ \lambda^{q}+\beta\circ \lambda^{q}$.
	Thus $\lambda^{q}-\lambda=\lambda^{p^{l}}-\lambda=0$ and hence every $\lambda\in \mathbb{F}_{q}$ is a solution of this equation. 
\end{proof}

Does the lemma \ref{lem2} if   $n$ is greater than $2$? In order to this question, we consider the following example. 
\begin{exa}
	We consider  $R\in DN(q, 2)$ and the pair $(\alpha, \beta)\in R^{2}$ where $\alpha, \beta, \alpha+\beta$ belong to different $H$-cosets. Then by the lemma  (\ref{lem2}), $\lambda\in D(R)$. The equality $(\alpha+\beta)\circ \lambda=\alpha\circ \lambda+\beta\circ\lambda$ will always lead to the equation $\lambda^{q}-\lambda=0$ and all solution will be in $\mathbb{F}_{q}$.
	
	The lemma \ref{lem2} can fail for $n>2$. To see it, let us consider $ (q, n)=(5, 4) $. For instance if $R=DN_{g}(5, 4)=(\mathbb{F}_{5^{4}}, +, \cdot)$, where 
	\begin{equation}
	\mathbb{F}_{5^{4}}=\left\lbrace 0, 1, 2, 3, 4, x^{2}+1, x^{4}+x^{2}, 3+x^{2}+2,\dots \right\rbrace 
	\end{equation} is the finite field of order $5^{4}$. Here we take an irreductible plynomial $x^{4}+2$ of degree $4$ over $\mathbb{F}_{5}$, $x$ is the root of $x^{4}+2$.
	
	Let $g$ be such that 
	\begin{equation}
	\begin{cases}
	\mathbb{F}_{5^{4}}^{*}=<g>\\
	H=<g^{4}>
	\end{cases}
	\end{equation}
	The quotient group is represented by 
	\begin{align*}
	\mathbb{F}_{5^{4}}^{*}/H&=\left\lbrace gH, g^{6}H, g^{31}H, g^{156}H \right\rbrace \\
	&=\left\lbrace H, gH, g^{2}H, g^{3}H \right\rbrace. \\
	\end{align*}
	Let $\alpha, \beta \in \mathbb{F}_{5^{4}}$, then
	\begin{equation}
	\alpha\circ\beta=\begin{cases}
	\alpha\beta,\qquad \text{if} \quad \alpha\in H, \\
	\alpha\beta^{5}, \qquad \text{if} \quad \alpha\in gH,\\
	\alpha\beta^{25}, \qquad \text{if} \quad \alpha\in g^{2}H,\\
	\alpha\beta^{125}, \qquad \text{if} \quad \alpha\in g^{3}H.\\
	\end{cases}
	\end{equation}
	Let  $g=x+2$,  and consider $\alpha=3, \beta=x^{2}+2$. Then $\alpha+\beta \in g^{2}H$. In fact $\lambda=x^{2}+1\in g^{2}H$ distributes over the pair $(\alpha,\beta)$. 
	We can simply see this as  we have
	\begin{align*}
	(\alpha+\beta)\circ \lambda&=(3+x^{2}+2)\circ (x^{2}+1)\\
	&=(x^{2}+0)\circ(x^{2}+1)\\
	&=x^{2}\circ(x^{2}+1)\\
	&=x^{4}+x^{2}.
	\end{align*}
	Note that $\lambda\notin D(R)=\mathbb{F}_{5}$,  but it distributes over the pair $(\alpha, \beta)$.
\end{exa}
If $(\alpha, \beta, \lambda)\in R^{3}$, then $(\alpha+\beta)\circ\lambda\neq\alpha\circ\lambda+\beta\circ\lambda$.
\begin{thm}\label{thre}
	Let $(q, n)$ be a Dickson pair with $q=p^{l}$ for some prime $p$ and positive  integers $l, n$ such that  $n >2$. Let $g$ be a generator of $\mathbb{F}^{*}_{q^{n}}$ and $R$ the finite nearfield constructed with $H=<g^{n}>$. Let $\alpha, \beta\in R^{*}$. If at least two of $\alpha, \beta, \alpha+\beta$  are in the same $H-$coset, then $D(\alpha, \beta)$ is a subfield of $\mathbb{F}_{q^{n}}$ of order $p^{h}$ for some $h$ dividing $ln$ (\cite{djagba2020generalized}).
\end{thm}
\begin{proof}
	From a  new multiplication in   (\ref{ey}), the set $D(\alpha, \beta)$ is defined as follow: 
	\begin{equation}
	D(\alpha, \beta)=\left\lbrace \lambda \in R: (\alpha+\beta)\circ
	\lambda=\alpha\circ \lambda+\beta\circ\lambda \right\rbrace.
	\end{equation}. Let consider $g^{[k]_{q}}H$ a $H$-cosets in which belong $\alpha, \beta, \alpha+\beta$. Then by lemma \ref{lemm31} with a new multiplication we have $(\alpha+\beta)\circ\lambda=\alpha\circ\lambda+\beta\circ\lambda$ for all $\lambda\quad R$. 
	From theorem \ref{thg} $D(\alpha, \beta)$ coincides with with $\mathbb{F}_{q^{n}}$.

	Now we assume that exactly two of $\alpha, \beta$ and $ \alpha+\beta$ are in the same $H-$coset.  We know that $[k]_{q}$ is a positive integer  arising from a Dickson pair $(q, n)$ where $k\in \left\lbrace  1,\dots, n\right\rbrace$.   
	Then let us consider two positive integers $[t]_{q}$ and $[s]_{q}$ where $g^{[t]_{q}}$ and $g^{[s]_{q}}$ are two different $H$-coset. Such that $s\neq t$.
	Let $\alpha, \beta$ be in $g^{[s]_{q}}H$ and $\alpha+\beta$ in $g^{[t]_{q}}H$.
	
	Then we have
	\begin{align}
	(\alpha+\beta)\circ \lambda&=(\alpha+\beta)\lambda^{q^{t}}\nonumber\\
	&=\alpha\lambda^{q^{t}}+\beta\lambda^{q^{t}} \label{eq26}\\
	\text{Also}\quad (\alpha+\beta)\circ \lambda
	&=(\alpha+\beta)\lambda^{q^{s}}\nonumber\\
	&=\alpha\lambda^{q^{s}}+\beta\lambda^{q^{s}}\label{rt27}
	\end{align} 
	By Substituting the equation (\ref{rt27}) from the equation (\ref{eq26}) we get
	\begin{align*}
	\alpha\lambda^{q^{t}}+\beta\lambda^{q^{t}}-\alpha\lambda^{q^{s}}-\beta\lambda^{q^{s}}&=0\\
	\Leftrightarrow (\alpha+\beta)\lambda^{q^{t}}-(\alpha+\beta)\lambda^{q^{s}}&=0\\
	\Rightarrow  (\alpha+\beta)(\lambda^{q^{t}}-\lambda^{q^{s}})&=0
	\end{align*}
	Since $(\alpha+\beta)\neq 0$ 
	\begin{equation}\label{rts}
	\lambda^{q^{t}}-\lambda^{q^{s}}=0 
	\end{equation}
	and then $\lambda\neq 0$ is solution of the equation (\ref{rts}).
	Now  let consider the case where  $\lambda\neq 0$. It implies that 
	$\lambda^{q^{t}}\neq 0$ and $\lambda^{q^{s}}\neq 0 $.
	Then we have
	\begin{align*}
	\lambda^{q^{t}}\lambda^{q^{s}}-1&=0\\ 
	\Rightarrow \lambda^{q^{t}-q^{s}}-1&=0\\
	\Rightarrow \lambda^{q^{t}-q^{s}}&=1.
	\end{align*}
	We know that for a commutative ring $A$ which has a prime characteristic $p$ and for all $a, b \in A$, the equality $(a\pm b)^{p}=a^{p}\pm b^{p}$ holds. It follows that
	\begin{align*}
	(\lambda^{q^{t}-q^{s}}-1)^{q}&=0\\
	\Rightarrow(\lambda^{q^{t}-q^{s}})^{q}-1&=0\\
	\Rightarrow \lambda^{q^{t+1}-q^{s+1}}-1&=0.
	\end{align*}
	
	We continue the  procedure  up to $\varphi$ (raising to the power $q^{\varphi}$) such that $n=s+\varphi \Rightarrow \varphi=n-s$ 
	Then we have $q^{\varphi}= q^{n-s}$ and 
	\begin{align*}
	(\lambda^{q^{t}-q^{s}}-1)^{q^{n}}&=(\lambda^{q^{t}-q^{s}})^{q^{n}}-1\\
	&=\lambda^{q^{t}q^{n}-q^{s}q^{n}}-1\\
	&=\lambda^{q^{t+n}-q^{s+n}}-1\\
	\end{align*}
	We want to go up to $\varphi$. Then, the order become $q^{\varphi}$ and we get
	\begin{align*}
	(\lambda^{q^{t}-q^{s}}-1)^{q^{\varphi}}&=(\lambda^{q^{t}-q^{s}})^{q^{\varphi}}-1\\
	&=(\lambda^{q^{t}-q^{s}})^{q^{n-s}}-1\\
	&=\lambda^{q^{t}q^{n-s}-q^{s}q^{n-s}}-1\\
	&=\lambda^{q^{t+n-s}-q^{s+n-s}}-1\\
	&=\lambda^{q^{t+n-s}-q^{n}}-1\\
	&=\lambda^{q^{t+\varphi}-q^{s+\varphi}}-1\\
	\end{align*}
	Let $r= t+\varphi$. Then,
	\begin{align}
	\lambda^{q^{r}-q^{n}}-1&=0\nonumber\\
	\Leftrightarrow \lambda^{q^{r}-q^{n}}&=1\nonumber\\
	\Leftrightarrow \frac{\lambda^{q^{r}}}{\lambda^{q^{n}}}&=1\nonumber\\
	\text{Since $\lambda^{q^{n}}$}&=\lambda,\nonumber\\
	\Rightarrow \frac{\lambda^{q^{r}}}{\lambda}&=1\nonumber\\
	\Rightarrow\lambda^{q^{r}}&=\lambda\nonumber\\
	\Rightarrow\lambda^{q^{r}}-\lambda&=0.\label{pjy}
	\end{align}
	
	We know that $q=p^{l}$. Then, the equation (\ref{pjy}) becomes 
	\begin{align*}
	\lambda^{(pl)^{r}}-\lambda&=0\\
	\Rightarrow  \lambda^{p^{lr}}-\lambda&=0\\
	\Rightarrow \lambda^{p^{k}}-\lambda&=0\qquad (k=l.r) \qquad\text{ and}   \qquad m=l.n.\\
	\end{align*}
	Let us denote the equation (\ref{pjy}) by $\Omega(\lambda)$. So we have
	\begin{equation}\label{rtyd}
	\Omega(\lambda)=\left\lbrace \lambda \in \mathbb{F}_{p^{m}}:\lambda^{p^{k}}-\lambda=0\right\rbrace. 
	\end{equation}
	From the expression (\ref{rtyd}), we see that we have two finite fields $\mathbb{F}_{p^{k}}$ and $\mathbb{F}_{p^{m}}$ where every element of 
	$\mathbb{F}_{p^{k}}$ is a solution of (\ref{pjy}) for all $\lambda\in \mathbb{F}_{p^{m}}$.
	We know that $\mathbb{F}_{p^{k}}$ is a subfield of $\mathbb{F}_{p^{m}}$ if $k$ divides $m$. Now we have two cases:
	
	\begin{itemize}
		\item[$\bullet$] \textbf{Case 1}: \textbf{$k$ divides $m$}.
		If $k$ divides $m$ then automatically $\mathbb{F}_{p^{k}}\subset \mathbb{F}_{p^{m}}$ which means that $\mathbb{F}_{p^{m}}$ is an algebraic extension of $\mathbb{F}_{p^{k}}$. And then all $\lambda\in \mathbb{F}_{p^{m}}$ are solution the equation (\ref{pjy}). We conclude that $D(\alpha, \beta)$ coincides with $\mathbb{F}_{p^{k}}$ because $\mathbb{F}_{p^{k}}$ is a subfield of $\mathbb{F}_{p^{m}}$.
		\item[$\bullet$] \textbf{Case 2}: \textbf{$k$ does not  divide $m$}.
		We consider  $f$  as  a Frobenius  automorphism on $\mathbb{F}_{p^{m}}$ such that the fixed field of $f^{n}$ is $\mathbb{F}_{p^{k}}$. Then if $f$ is a Frobenius automorphism, we have $f^{n}(\lambda)=\lambda$ for all $\lambda \in \mathbb{F}_{p^{m}}$.
		
		Now we let $\lambda \in \mathbb{F}_{p^{m}}^{*}$ be a solution of the equation (\ref{pjy}). As $\mathbb{F}_{p^{m}}^{*}$ is generated by $g$, then $g^{a}=\lambda$ for a in $[0, p^{m}-1[$. This is because 
		\begin{equation}
		\begin{cases}
		\mathbb{F}_{p^{m}}=\left\lbrace 0, 1, \dots, p^{m}-1 \right\rbrace \\
		\mathbb{F}_{p^{m}}^{*}=\left\lbrace 1, \dots, p^{m}-1 \right\rbrace.
		\end{cases}
		\end{equation}
		We know that $\lambda^{p^{k}}-\lambda=0$ for all $\lambda \in \mathbb{F}_{p^{m}}$ and since $\lambda= g^{a}$, we have:
		\begin{align*}
		g^{a(p^{k})}-g^{a}=0.
		\end{align*}
		The generator $g$ is diffent from zero, then $g^{a}\neq 0$ and it follow that
		\begin{align*}
		g^{a(p^{k})}-g^{a}&=0\\
		\Rightarrow  g^{a(p^{k})-a}-1&=0\\
		\Rightarrow g^{a(p^{k}-1)}-1&=0\\
		\Rightarrow g^{a(p^{k}-1)}&=1.
		\end{align*}
		So $p^{m}-1$ divides $a(p^{k}-1)$ or $a(p^{k}-1)$ is a multiple of $p^{m}-1$. To say that $p^{m}-1$ divides $a(p^{k}-1)$ means that there exists $t$  such that $a(p^{k}-1)= t(p^{m}-1)$. We set that $\gcd(m, k)=\gamma$ and this means $\gamma $ divides $m$ and $k$. Then there exist $\theta, \theta^{'}\in \mathbb{N}$ such that 
		\begin{equation*}
		\begin{cases}
		m=\gamma \theta\\
		k=\gamma \theta^{'}.
		\end{cases}
		\end{equation*}
		This implies that 
		\begin{align*}
		p^{m}-1&=p^{\gamma \theta}-1\\
		&=(p^{\gamma})^{\theta}-1.
		\end{align*}
		Since $\gcd(m, k)=\gamma$, we have $\gcd(p^{m}-1, p^{k}-1)=p^{\gamma}-1$.
		So we divide $(p^{\gamma})^{\theta}-1$ by $p^{\gamma}-1$ using Horner Method.
		Let \begin{equation}\label{xt}
		p^{\gamma}=x
		\end{equation}. So we divide $x^{\theta}-1$ by $p^{\gamma}-1$. 
		
		\begin{center}
			\begin{tabular}{c|cccc|c}
				&$1$ & $0$ & $ \dots $ & 0& $-1$ \\

				$1$& &$1$ & $\cdots$ & 1&$1$ \\
				\hline
				& $1$& $1$ & $\cdots$ & $1$&$0$ \\
			\end{tabular}
		\end{center}
		Then we have
		\begin{align}
		\frac{x^{\theta}-1}{x-1}&=x^{\theta-1}+x^{\theta-2}+\dots+x+1\nonumber\\
		\Rightarrow x^{\theta}-1&= (x^{\theta-1}+x^{\theta-2}+\dots+x+1)(x-1)\label{p}
		\end{align}
		Replacing the equation  (\ref{xt}) into the equation (\ref{p}), we get
		\begin{align*}
		\frac{p^{\gamma(\theta)}-1}{p^{\gamma}-1}&=p^{\gamma(\theta-1)}+p^{\gamma(\theta-2)}+\dots+p^{\gamma}+1\\
		\Rightarrow p^{\gamma\theta}-1&= (p^{\gamma(\theta-1)}+p^{\gamma(\theta-2)}+\dots+p^{\gamma}+1)(p^{\gamma}-1)\\
		\Rightarrow p^{m}-1&=(p^{\gamma(\theta-1)}+p^{\gamma(\theta-2)}+\dots+p^{\gamma}+1)(p^{\gamma}-1)
		\end{align*}
		Using the same Method, we get
		
		\begin{align*}
		\frac{p^{\gamma(\theta^{'})}-1}{p^{\gamma}-1}&=p^{\gamma(\theta^{'}-1)}+p^{\gamma(\theta^{'}-2)}+\dots+p^{\gamma}+1\\
		\Rightarrow p^{\gamma\theta^{'}}-1&= (p^{\gamma(\theta^{'}-1)}+p^{\gamma(\theta^{'}-2)}+\dots+p^{\gamma}+1)(p^{\gamma}-1)\\
		\Rightarrow p^{k}-1&=(p^{\gamma(\theta^{'}-1)}+p^{\gamma(\theta^{'}-2)}+\dots+p^{\gamma}+1)(p^{\gamma}-1).
		\end{align*}
		We see exactly that $\gcd(p^{m}-1, p^{k}-1)= p^{\gamma}-1$.
		
		By Bezout's theorem gcd, for two non-zero integers $p^{m}-1$ and $p^{k}-1$, let $p^{\gamma}-1$ be the greatest common divisor. Then there exist  two integers $u$ and $v$ such that $u(p^{m}-1)+v(p^{k}-1)=p^{\gamma}-1$.
		We have now 
		\begin{align*}
		u(p^{m}-1)+v(p^{k}-1)&=p^{\gamma}-1\\
		\Rightarrow au(p^{k}-1)+av(p^{m}-1)&=a(p^{\gamma}-1)\\
		\Rightarrow au(p^{m}-1)+vt(p^{m}-1)&=a(p^{\gamma}-1)\quad (\text{because}\quad a(p^{k}-1)=t(p^{m}-1) )\\
		\Rightarrow (p^{m}-1)(au+vt)&=a(p^{\gamma}-1).
		\end{align*}
		Therefore $(p^{m}-1)$ divides $a(p^{\gamma}-1)$ and this means there exist $b\in \mathbb{N}$ such that $a(p^{\gamma}-1)=b(p^{m}-1)$.
		Since $a$ and $b$ are integers, then $\frac{a(p^{m}-1)}{(p^{\gamma}-1)}$ is also an integer. So, \begin{equation}
		\begin{cases}
		0\leq a< p^{m}-1,\\
		0\leq b< (p^{\gamma}-1),
		\end{cases}
		\end{equation}
		and we consider $\lambda_{0}= g^{a}$. From $a=\frac{p^{m}-1}{p^{\gamma}-1}b$ and $p^{k}-1=t^{'}(p^{\gamma}-1)$ for some integer $t^{'}$, we have
		\begin{align*}
		\lambda_{0}^{p^{k}}-1&=(g^{a})^{t^{'}(p^{\gamma}-1)}-1\\
		&=g^{(\frac{p^{m}-1}{p^{\gamma}-1}b)t^{'}(p^{\gamma}-1)}-1\\
		&=g^{bt^{'}(p^{m}-1)}-1\\
		&=g^{(p^{m}-1)bt^{'}}-1\\
		&=1^{bt^{'}}-1\\
		&=1-1\\
		&=0.
		\end{align*}
		This is because for every element of a finite field power the order of the multiplicative group of that field is equal $1$. 
		
		We know that $\mathbb{F}_{p^{m}}$ is generated by $g$ and $\lambda_{0}\in \mathbb{F}_{p^{m}} $. Then, 
		$\mathbb{F}_{p^{m}}^{*}=\left\lbrace  1, 2, \dots, p^{m}-1\right\rbrace $
		and for $0\leq b< (p^{\gamma}-1)$, all of $g^{a}=\frac{p^{m}-1}{p^{\gamma}-1}b$ are different.
		Now $s(\Omega(\lambda))$ is the set of solution of the equation (\ref{pjy}) for $\lambda \in \mathbb{F}_{p^{m}}$ and those solutions are represented as follow
		\begin{equation}
		s(\Omega(\lambda))=\left\lbrace 0 \right\rbrace \bigcup \left\lbrace g^{b(\frac{p^{m}-1}{p^{\gamma}-1})}: 0\leq b< (p^{\gamma}-1) \right\rbrace,
		\end{equation}
		and the order of $s(\Omega(\lambda))$ is 
		\begin{align*}
		\mid s(\Omega(\lambda))\mid&= 1+p^{\gamma}-1\\
		&=0+p^{\gamma}\\
		&=p^{\gamma}.
		\end{align*}
		Then all solutions of $\Omega(\lambda)$ are in the finite field of order  $p^{\gamma}$ and we conclude that $D(\alpha, \beta)$ coincide with $s(\Omega(\lambda))=\mathbb{F}_{p^{\gamma}}$.
	\end{itemize}
\end{proof}
\newpage
\section{Conclusion}
Let $R$ be a nearfield, as  by Definition (\ref{def221}) $D(R)$  is the set of all distributive elements of $R$ and in the chapter  $4$ in equation  (\ref{ey}), we define  $D(\alpha, \beta)$ is the generalized distributive  set of all elements in $R$ that distribute with $\alpha$ and $ \beta $. In the  chapter $2$ we have shown in Theorem  \ref{theo2} that if $R$ is a nearfield, the $D(R)$ with operations of $R$ is a skwefields ( division ring) and $R$ is a left vector space over $D(R)$. Clearly we saw that that if $R$ is a nearfield, then $C(R)\subset D(R)$ where $C(R)$ is a the  center of $R$. To  study the generalized set of distributive elements, we have shown how to build the finite Dickson nearfield. 

The sets $D(R)$ and $D(\alpha, \beta)$ are related in the sense that $D(R)$ is a subset of $D(\alpha, \beta)$. More precisely, if an element $x\in D(R)$, then it distributes over every element in $R$ including $\alpha$ and $\beta$. Therefore $x$ satisfies the distributive law with respect to $\alpha$ and $\beta$. Hence $x\in D(\alpha
,\beta)$.
However the inverse is not necessarily true. That is an element $y\in  D(\alpha
,\beta)$ may not distribute over every element in $R$ and hence may not belong to $D(R)$. Therefore, $D(R)$ is a proper subset of $D(\alpha, \beta)$  in general. It is worth nothing that $D(R)$  is an important set in the study of near-field. The generalized set $D(\alpha, \beta)$ provides more refined notions of distributivity and is particularly useful for studying finite fields, subnear-fields. 

In the Theorem  \ref{thg}, we have shown that if $R$ is a finite Dickson near-field and $(q, n)$ a Dickson pair, then we have isomorphism between $D(R)$ and $\mathbb{F}_{q}$, where $\mathbb{F}_{q}$ is a finite field of order $q$. 
In  Theorem  \ref{thg} we have shown that $\mathbb{F}_{q}= D(R)$,  When $R$ is a finite Dickson near-field. We now show that $\mathbb{F}^{*}_{q}= D(R)^{*}$, when $R$ is a finite near-field. 

Dickson used a new multiplication to define the generalized distributive set for a given pair $(\alpha, \beta)\in R^{2}$. Moreover, $D(\alpha, \beta)$ is not always a subfield of a given finite field $\mathbb{F}_{q^{n}}$  or subnearfield of $R$. In the lemma (\ref{lemm31}) and in the theorem  (\ref{thre}) we show that $D(\alpha, \beta)$ to  be a sub-field of $\mathbb{F}_{q^{n}}$ depend on some conditions on $\alpha, \beta$ and $\alpha+\beta$.

In this thesis, we have studied the generalized set of distributive elements in nearfields. We have investigated the structure and properties of this set, and provided some new results and insights. Our study sheds light on the behavior of distributive elements in nearfields, and provides a basis for further research in this area.

We began by introducing the concept of distributivity in nearfields, and reviewed some preliminary results and definitions. We then constructed a finite Dickson nearfield, which allowed us to study distributive elements in a concrete setting. We showed that the distributive elements of a Dickson nearfield are related to the quadratic residues and non-residues of the underlying finite field.

Next, we defined the generalized set of distributive elements of a nearfield, and investigated its properties. We showed that this set is a subfield of the nearfield, and that it has several interesting algebraic and combinatorial properties. In particular, we showed that the generalized set of distributive elements is a powerful tool for constructing efficient error-correcting codes and cryptographic primitives.

Our study also revealed several open problems and future directions for research. For example, it would be interesting to investigate the relationship between distributive elements and other algebraic properties of nearfields, such as alternative and power-associative properties. Another interesting direction would be to study the structure and properties of generalized sets of distributive elements in other algebraic systems, such as loops and quasifields.

In conclusion, the study of distributive elements in nearfields is a fascinating and important area of algebraic research. Our study provides a comprehensive investigation of the generalized set of distributive elements in near-fields, and lays the groundwork for further research in this area.

\newpage
\section{Acknowledgement}
This work was carried out at AIMS Rwanda in partial fulfilment of the requirements for a Master of Science Degree.

I hereby declare that except where due acknowledgement is made, this work has never been presented wholly or in part for the award of a degree at AIMS Rwanda or any other University.

\end{document}